\documentclass[12pt]{article}

\newcommand{\Aut}{\mathrm{Aut}}

\newcommand{\ad} {\mathrm {ad}}
\newcommand{\id} {\mathrm {id}}

\newcommand{\End} {\mathrm {End}}

\usepackage{amssymb}
\usepackage{amsthm}
\usepackage{amsmath}
\newtheorem{theorem}{Theorem}[section] 
\newtheorem{lemma}[theorem]{Lemma}     
\newtheorem{corollary}[theorem]{Corollary}

\newtheorem{definition}{Definition}

\newtheorem{notation}{Notation} 

\begin{document}

\title{Associative Subalgebras of Low-Dimensional Majorana Algebras} 

\author{A. Castillo-Ramirez\thanks{Funded by the \emph{Universidad de Guadalajara} and an Imperial College International Scholarship.} \\
\small{Imperial College London, Department of Mathematics.} \\
\small{South Kensington Campus, London, SW7 2AZ.} \\
\small{Email address: \texttt{ac1209@imperial.ac.uk}} }

\date{May 2014}

\maketitle

\begin{abstract}
A \emph{Majorana algebra} is a commutative nonassociative real algebra generated by a finite set of idempotents, called \emph{Majorana axes}, that satisfy some of the properties of the $2A$-axes of the Monster Griess algebra. The term was introduced by A. A. Ivanov in 2009 inspired by the work of S. Sakuma and M. Miyamoto. In the present paper, we revisit Mayer and Neutsch's theorem on associative subalgebras of the Griess algebra in the context of Majorana theory. We apply this result to determine all the maximal associative subalgebras of some low-dimensional Majorana algebras; namely, the Majorana algebras generated by two Majorana axes and the Majorana representations of the symmetric group of degree $4$ involving $3C$-algebras. \\

\textbf{Keywords:} Majorana representation, Monster group, Griess algebra.
\end{abstract}

\section{Introduction} \label{intro}

The largest of the sporadic simple groups, the Monster group $\mathbb{M}$, was constructed by Griess \cite{G82} as a group of automorphisms of a $196,884$-dimensional commutative nonassociative real algebra $V_{\mathbb{M}}$. 

One of the major obstacles in the examination of the Griess algebra is its nonassociativity. A natural approach to deal with this issue is the study of the associative subalgebras of $V_{\mathbb{M}}$; the seminal work in this direction was done by Meyer and Neutsch \cite{MN93}, who proved that every maximal associative subalgebra of $V_{\mathbb{M}}$ has an orthogonal basis of indecomposable idempotents. Furthermore, they conjectured that $48$ is the largest possible dimension of an associative subalgebra of $V_{\mathbb{M}}$. By showing that the length of any idempotent of $V_{\mathbb{M}}$ is at least $1$ (with respect to our scaling), Miyamoto \cite{Mi96} proved this conjecture. After this, connections between Virasoro frames, root systems, Niemeier lattices and maximal associative subalgebras of $V_{\mathbb{M}}$ have been explored in \cite{DLMN98}.      

Inspired by Sakuma's theorem \cite{S07}, Ivanov \cite[Ch.~8]{I09} introduced the following terminology. Suppose that $V$ is a commutative nonassociative real algebra with inner product generated by a finite set $\Omega$ of idempotents of length $1$. We say that $V$ is a \emph{Majorana algebra} with \emph{Majorana axes} $\Omega$ if axioms \textbf{M1-M7} of \cite[p.~2433]{IPSS10} are satisfied (with $\Omega$ in the place of $A$). The \emph{Majorana involutions} of $V$, which correspond to the restriction of the \emph{Miyamoto involutions} \cite{Mi96}, are the automorphisms of $V$ defined by axiom \textbf{M6}. If $G$ is a finite group generated by a set of involutions $T$ such that $T^g=T$, for all $g \in G$, a \emph{Majorana representation} of $(G,T)$ on a Majorana algebra $V$ is a linear representation of $G$ on $V$ together with a compatible injection of $T$ into the Majorana axes of $V$ (see \cite[Sec.~3]{IPSS10}). We denote by $a_t$ the Majorana axis corresponding to $t \in T$. 

The prototypical example of a Majorana algebra is $V_{\mathbb{M}}$ (see \cite[Ch.~8]{I09}). In \cite{IPSS10}, Sakuma's theorem was proved for Majorana algebras: this establishes that the Majorana algebras generated by two Majorana axes have eight possible isomorphism types. Majorana representations of various finite groups have been described in \cite{IPSS10}, \cite{IS12}, \cite{I11b}, \cite{I11a} and \cite{IS}. Notably, Seress \cite{Se12} designed and implemented an algorithm for the construction of $2$-closed Majorana representations.

In this paper, we discuss Meyer and Neutsch's theorem \cite{MN93} in the context of Majorana algebras, and we apply it to some low-dimensional cases. In particular, we find all the maximal associative subalgebras of the Majorana algebras generated by two Majorana axes using the information on idempotents obtained in \cite{CR13, LYY05}. Furthermore, we obtain all the idempotents and maximal associative subalgebras of the Majorana representations of $S_4$ of shapes $(2B,3C)$ and $(2A,3C)$ described in \cite[Sec.~4]{IPSS10}. 

In the context of Vertex Operator Algebras (VOAs), the sets of idempotents generating maximal associative subalgebras of $V_{\mathbb{M}}$ are relevant because they determine distinct \emph{Virasoro frames}, which are sets of vectors in the Moonshine module that generate framed sub-VOAs (see \cite{DGH98}). Few of the maximal associative subalgebras described in this paper have been previously obtained in \cite[Appendix~A]{LYY05} in the VOA context. 

%%%% SEC1
\section{Associative subalgebras of Majorana representations} \label{sec:1}

Let $V=(V,\;.\;,(\;,\;))$ be a commutative nonassociative real algebra with inner product. For $S \subseteq V$, denote by $\langle\langle S \rangle\rangle$ the smallest subalgebra of $V$ containing $S$. If $\langle\langle S \rangle\rangle = V$, we say that $S$ generates the algebra $V$. If $v \in V$, the \emph{length} of $v$ is the non-negative real number $l(v):=(v,v)$. Say that $v \in V$ is \emph{idempotent} if $v \cdot v = v$.

Throughout this section, we assume that $V$ is a Majorana algebra with identity $\id \in V$. As we shall constantly use axioms \textbf{M1} and \textbf{M2} of \cite[p.~2433]{IPSS10}, we rewrite them below:

 \begin{description}
\item \textbf{M1} The inner and algebra products associate in $V$: $( u,v\cdot w) = ( u\cdot v,w)$, for every $u,v,w\in V$.
\item \textbf{M2} The \emph{Norton inequality} holds in $V$: $( u\cdot u,v\cdot v ) \geq ( u\cdot v,u\cdot v )$, for every $u,v\in V$. 
\end{description}

Define the \emph{adjoint transformation} of $v \in V$ as the endomorphism $ad_v$ of $V$ that maps $u \in V$ to $v \cdot u$.

\begin{lemma} \label{semisimple}
For any $v \in V$, the adjoint transformation $\ad_v \in \End(V)$ is orthogonally diagonalisable (i.e. there is an $(\;,\;)$-orthogonal basis of $V$ of eigenvectors of $\ad_v$).
\end{lemma}
\begin{proof}
By \textbf{M1}, the adjoint transformation of $v \in V$ is a self-adjoint endomorphism of $V$, so the result follows by the Spectral Theorem (see \cite[Sec.~10]{R08}).
\end{proof}

The following is another elementary consequence of \textbf{M1}.

\begin{lemma}\label{linearity}
Let $\{ x_i \in V : 1 \leq i \leq k \}$ be a finite set of idempotents. Let $\lambda_i \in \mathbb{R}$ and suppose that
\[ x = \sum \limits_{i=1}^k \lambda_i  x_i \]
is also an idempotent. Then,
\[ l(x) = \sum \limits_{i=1}^k \lambda_i l(x_i). \]
\end{lemma} 
\begin{proof}
The lemma follows by \textbf{M1} and the linearity of the inner product:
\[ l(x) = (x,x) =  (x \cdot x, \id ) = (x,\id) = \sum \limits_{i=1}^k \lambda_i (x_i, \id) = \sum \limits_{i=1}^k \lambda_i l(x_i). \] 
\end{proof}

\begin{lemma}\label{orthidem}
Let $x,y \in V$ be idempotents. The following are equivalent:
\begin{description}
\item[(i)] $(x,y) = 0$.
\item[(ii)] $x \cdot y =0$.
\item[(iii)] $x+y$ is idempotent.
\end{description}
\end{lemma} 
\begin{proof}
If $(x,y) = 0$, \textbf{M2} implies that
\[ (x \cdot y, x \cdot y) \leq (x \cdot x, y \cdot y) = (x,y) =0. \]
Hence $x \cdot y = 0$ by the positive-definiteness of the inner product. It is clear that part \textbf{(ii)} implies \textbf{(iii)}. Lemma \ref{linearity} shows that \textbf{(iii)} implies \textbf{(i)}. 
\end{proof}

We say that two idempotents are \emph{orthogonal} if they satisfy the equivalent statements \textbf{(i)-(iii)} of Lemma \ref{orthidem}. 
\begin{notation}
For $\mu \in \mathbb{R}$ and $v \in V$, denote by $V_\mu^{(v)}$ the $\mu$-eigenspace of $\ad_v$, and
\[ d_\mu(v) := \dim(V_\mu^{(v)}). \]
\end{notation}
We say that $\mu \in \mathbb{R}$ is an eigenvalue of $v \in V$ if $\mu$ is an eigenvalue of the adjoint transformation of $v$. Let $\Aut(V)$ be the set of distance-preserving endomorphisms $\phi$ of $V$ such that $v^\phi \cdot u^\phi = (v \cdot u)^\phi$, for all $u,v \in V$.

\begin{lemma} \label{spectrum}
Let $x \in V$ be a non-zero non-identity idempotent. Then:
\begin{description}
\item[(i)] $d_\mu(x) \neq 0$, for $\mu \in \{0,1\}$.

\item[(ii)] For every $g \in \Aut(V)$, the spectrum of $x^{g}$ is equal to the spectrum of $x$.

\item[(iii)] If $\{ \mu_i : 1 \leq i \leq k \}$ is the spectrum of $x$, then $\{ 1 - \mu_i : 1 \leq i \leq k \}$ is the spectrum of the idempotent $\id-x$.
\end{description}
\end{lemma}
\begin{proof}
Part \textbf{(i)} follows because $x$ and $\id - x$ are $1$- and $0$-eigenvectors of $\ad_x$, respectively. Part \textbf{(ii)} follows since, for any $g \in \Aut(V)$, $v$ is an eigenvector of $\ad_x$ if and only if $v^g$ is an eigenvector of $\ad_{x^g}$. Part \textbf{(iii)} is straightforward.
\end{proof}

The main results of this paper hold for Majorana algebras that satisfy the following stronger version of axiom \textbf{M2}:
\begin{description}
\item ${\bf M2}'$ The Norton inequality holds for every $u,v\in V$, with equality precisely when the adjoint transformations $\ad_u$ and $\ad_v$ commute.
\end{description}

It was shown in \cite[Sec.~15]{C84} that ${\bf M2}'$ holds in $V_{\mathbb{M}}$.  The importance of ${\bf M2}'$ in our discussion lies in the following lemma.

\begin{lemma}\label{sub}
Let $V$ be a Majorana algebra that satisfies ${\bf M2}'$. Let $x \in V$ be an idempotent and $\mu \in \{0,1\}$. Then:
\begin{description}
\item[(i)] The eigenspace $V_\mu^{(x)}$ is a subalgebra of $V$.

\item[(ii)] The number of idempotents in $V_\mu^{(x)}$ is either infinite or at most $2^{d_{\mu}(x)}$. 
\end{description}
\end{lemma}
\begin{proof}
Let  $y \in V_\mu^{(x)}$. Observe that
\[ (x \cdot x, y \cdot y) = (x \cdot y, y) = (\mu y, y) = (\mu y , \mu y) = (x \cdot y, x \cdot y). \]
Therefore, ${\bf M2}'$ implies that $\ad_x$ and $\ad_{y}$ commute, and so, for any $y' \in V_\mu^{(x)}$,
\[ x \cdot (y \cdot y') = y \cdot (x \cdot y') = \mu (y \cdot y'). \]
This shows that $(y \cdot y') \in V_\mu^{(x)}$. Part \textbf{(i)} follows.

As $V_\mu^{(x)}$ is a subalgebra of $V$, its idempotents are in a bijective correspondence with the solutions of the system of $d_\mu(x) \times d_\mu(x)$ quadratic equations determined by the coordinates of the relation $v \cdot v - v=0$. Therefore, if the system has finitely many solutions, B\'ezout's theorem implies that $V_\mu^{(x)}$ has at most $2^{d_\mu(x)}$ idempotents. Part \textbf{(ii)} follows. 
\end{proof}

An idempotent is \emph{decomposable} if it may be expressed as a sum of at least two non-zero idempotents; otherwise, the idempotent is \emph{indecomposable}.

The following results were obtained in \cite{MN93} with $V=V_{\mathbb{M}}$; however, their proof applies, without change, to Majorana algebras that satisfy ${\bf M2}'$.

\begin{theorem}[Meyer, Neutsch]\label{indecomposable}
Let $V$ be a Majorana algebra with identity that satisfies ${\bf M2}'$. An idempotent $x \in V$ is indecomposable if and only if $d_1(x)=1$.
\end{theorem}

\begin{theorem}[Meyer, Neutsch] \label{Meyer-Neutsch}
Let $V$ be a Majorana algebra with identity that satisfies ${\bf M2}'$. Let $U$ be a finite-dimensional subalgebra of $V$. Then:
\begin{description}
\item[(i)] $U$ is associative if and only if $U$ has an orthogonal basis of idempotents.

\item[(ii)] $U$ is maximal associative if and only if $\id \in U$ and $U$ has an orthogonal basis of indecomposable idempotents.
\end{description}
\end{theorem}

Notice that if $B:=\{ x_i \in V : 1 \leq i \leq k\}$ is a set of pairwise orthogonal idempotents, then the associative subalgebra $U := \langle \langle B \rangle \rangle$ coincides with the linear span of $B$; in particular, the dimension of $U$ is $k$. If the idempotents of $B$ are indecomposable, all the non-zero idempotents of $U$ are finite sums of the idempotents of $B$; in particular, $U$ may have at most one orthogonal basis of indecomposable idempotents. 

When $x \in V$ is a non-zero non-identity idempotent, $\{ x, \id - x\}$ is an orthogonal basis of a two-dimensional associative subalgebra 
\[ V_x: = \langle \langle x, \id - x \rangle \rangle \leq V. \]
Define $V_0 := \langle\langle 0 \rangle\rangle$ and $V_{\id} := \langle\langle \id \rangle\rangle$.

\begin{definition}
We say that an associative subalgebra $U$ of $V$ is \emph{trivial associative} if $U = V_x$ for some idempotent $x \in V$.  
\end{definition}

\begin{corollary}\label{trivial max}
Let $x \in V$ be an idempotent. The trivial associative subalgebra $V_x \leq V$ is maximal associative if and only if $d_0(x) = d_1(x)=1$.
\end{corollary}
\begin{proof}
The result follows by Lemma \ref{spectrum} (iii), Theorems \ref{indecomposable}  and \ref{Meyer-Neutsch} (ii). 
\end{proof}

\begin{corollary} \label{cormax}
Let $B:=\{x_i \in V : 1 \leq i \leq k \}$ be a finite set of indecomposable idempotents. The subalgebra $U := \langle \langle B \rangle \rangle$ is maximal associative if and only if $\id = \sum \limits_{i=1}^k x_i$.
\end{corollary}
\begin{proof}
If $U$ is maximal associative, then $\id = \sum \limits_{i=1}^k x_i$ by Theorem \ref{Meyer-Neutsch} \textbf{(ii)} and the uniqueness of the identity in a commutative algebra. Suppose $\id = \sum \limits_{i=1}^k x_i$. Note that \textbf{M2} and the positive-definiteness of the inner product imply that $(x,y) \geq 0$, for every pair of idempotents $x,y \in V$. By Lemma \ref{linearity}, 
\[ l(\id) = \sum \limits_{i=1}^k l(x_i) = \sum \limits_{i=1}^k l(x_i) + \sum \limits_{i \neq j} (x_i, x_j). \] 
Therefore $(x_i, x_j)=0$, for any $i \neq j$. Hence, $U$ is maximal associative by Theorem \ref{Meyer-Neutsch} \textbf{(ii)}.
\end{proof}

The following lemmas will be useful in our discussion of low-dimensional Majorana algebras.

\begin{lemma}\label{count assoc}
Let $x \in V$ be an indecomposable idempotent and suppose that $V_0^{(x)}$ has finitely many idempotents. Then:
\begin{description}
\item[(i)] If $d_0(x)=1$, then $x$ is not contained in any three-dimensional associative subalgebra of $V$.

\item[(ii)]  If $d_0(x) \geq 2$, then $x$ is contained in at most $2^{d_0(x)-1} - 1$ three-dimensional maximal associative subalgebras of $V$. 
\end{description}
\end{lemma}
\begin{proof}
Part \textbf{(i)} follows by Corollary \ref{cormax}. Let $d_0(x) \geq 2$ and let $U \leq V$ be a three-dimensional maximal associative subalgebra such that $x \in U$. By Theorem \ref{Meyer-Neutsch} \textbf{(ii)}, we may find an orthogonal basis of $U$ of indecomposable idempotents: $\{ x, y, z \}$, where $y,z \in V_0^{(x)}$. If there is an idempotent $w \in V_0^{(x)}$ such that $\langle\langle x, y, w \rangle\rangle$ is maximal associative, then Theorem \ref{Meyer-Neutsch} \textbf{(ii)} implies that
\[ x + y + z = \id = x + y + w, \]
so $z = w$. This shows that distinct three-dimensional maximal associative subalgebras of $V$ containing $x$ correspond to distinct disjoint two-sets of non-zero idempotents of $V_0^{(x)}$. Lemma \ref{sub} \textbf{(ii)} implies that  there are at most $\frac{1}{2} (2^{d_0(x)} - 2)$ disjoint two-sets of non-zero idempotents in $V_0^{(x)}$. 
\end{proof}

\begin{lemma}\label{max4}
Suppose that, for every idempotent $x \in V$, the space $V_0^{(x)}$ has finitely many idempotents, and $d_0(x) \leq 2$. Then the dimension of every associative subalgebra of $V$ is at most three.
\end{lemma}
\begin{proof}
If $\{ x, y, z, w \}$ is a set of four pairwise idempotents of $V$, then $x$ is orthogonal to $7$ non-zero idempotents: $y$, $z$, $w$, $y+z$, $y+w$, $z+w$ and $y+z+w$. However, as $V_0^{(x)}$ has finitely many idempotents, and $d_0(x) \leq 2$, Lemma \ref{sub} \textbf{(ii)} implies that $x$ is orthogonal to at most $2^{d_0(x)} -1 \leq 3$ non-zero idempotents. The result follows.
\end{proof}

If $U$ is a subalgebra of $V$, we introduce the notation
\[ [U] : = \{ U^g : g \in \Aut(V) \} . \] 
Because of Lemma \ref{semisimple}, we know that, for every $v \in V$, the geometric and algebraic multiplicities of the eigenvalues of $\ad_v$ coincide. In the following sections, we calculate the spectrum of every idempotent of some low-dimensional Majorana algebras in order to obtain all their maximal associative subalgebras. Because of Lemma \ref{spectrum} \textbf{(ii)}, we organise these spectra in terms of $\Aut(V)$-orbits. We consider only half of the non-zero non-identity idempotents, as the spectra of the remaining idempotents may be found using Lemma \ref{spectrum} \textbf{(iii)}. Although we require only the multiplicities of $0$ and $1$, we believe that the full spectrum of each idempotent may be of general interest. 

%%%%%%%%%%%%%%%%%%%%%%%%%%%%%

\section{The Norton-Sakuma Algebras} \label{NS-algebras assoc}

If $t,g$ are $2A$-involutions of $\mathbb{M}$, it is known (see \cite{C84,GMS89}) that the product $tg$ lies in one of the following conjugacy classes of $\mathbb{M}$:
\[ 1A, \ 2A, \ 2B, \ 3A, \ 3C, \ 4A, \ 4B, \ 5A, \ 6A. \]
In \cite{C84,N96}, Norton described the subalgebras of $V_{\mathbb{M}}$ generated by the $2A$-axes corresponding to $t$ and $g$, and he labelled them according to the conjugacy class of $tg$.
 
Let $V$ be a Majorana algebra. Say that $V$ is a \emph{Norton-Sakuma algebra} if $V$ is generated by exactly two of its Majorana axes. As proved in \cite{IPSS10}, Sakuma's theorem implies that the Norton-Sakuma algebras coincide with the non-trivial subalgebras of $V_{\mathbb{M}}$ described by Norton. In this context, the Norton-Sakuma algebras are described in \cite[Table~3]{IPSS10}.

Denote by $V_{NX}$ the Norton-Sakuma algebra of type $NX$. Denote by $\Aut(NX)$ and $\id_{NX}$ the automorphism group and identity of $V_{NX}$, respectively. The groups $\Aut(NX)$ are described explicitly in \cite[Theorem~4.1]{CR13}, while the identities $\id_{NX}$ are given by \cite[Table~2]{CR13}. Note that, except for $V_{2B}$, none of the Norton-Sakuma algebras is associative. As $V_{NX}$ may be embedded in $V_{\mathbb{M}}$, the Norton-Sakuma algebras clearly satisfy axiom ${\bf M2}'$.

The complete list of idempotents of all the Norton-Sakuma algebras was obtained in \cite{CR13}. We shall constantly use this information to prove the results of this section. 

%%%%%
\subsection{Associative Subalgebras of $V_{2A}$, $V_{3A}$ and $V_{3C}$}

The following result examines the Norton-Sakuma algebras of types $2A$, $3A$ and $3C$. 

%PROPOSITION 2A, 3A AND 3C
\begin{lemma}
Let $NX \in \{2A, 3A, 3C\}$. A subalgebra of $V_{NX}$ is maximal associative if and only if it is trivial associative.
\end{lemma}
\begin{proof}
All the idempotents of $V_{NX}$ are described in \cite[429--431]{CR13}. Direct calculations show that $d_0(x) = d_1(x) = 1$ for every non-zero non-identity idempotent $x \in V_{NX}$. The result follows by Corollary \ref{trivial max}.
\end{proof}

%% CASE 4A %%%%%%%%%%%%%%%%%%%%%%
\subsection{Associative Subalgebras of $V_{4A}$}

Let $\{a_0, a_1, a_{-1}, a_2, v_{\rho}\}$ be the basis of $V_{4A}$ as in \cite[Table~3]{IPSS10}. It was shown in \cite[Sec. 3.2]{CR13} that $V_{4A}$ has an infinite family of idempotents of length $2$: for any $\lambda \in [  -\frac{3}{5},1]$, there is an idempotent $y^{(1)}_{4A} ( \lambda ) \in V_{4A}$. Direct calculations show that the spectrum of  $y^{(1)}_{4A} ( \lambda )$ is
\[ \{ 0, \ 1, \ \frac{1}{2}, \ h\left( \lambda \right) , \ \overline{h\left( \lambda \right) }\},  \] 
where $\overline{h\left( \lambda \right) }$ is the conjugate in $\mathbb{Q}\left( \sqrt{-15\lambda ^{2}+6\lambda +9}\right) $ of 
\[ h(\lambda) := \frac{1}{2^5} ( 17 - 5 \lambda -  5 \sqrt{-15\lambda^2 + 6 \lambda + 9}). \]  

\begin{lemma}\label{infinite}
Consider the infinite family $\{y^{(1)}_{4A} ( \lambda ) : \lambda \in [  -\frac{3}{5},1] \}$ of idempotents of $V_{4A}$.
\begin{description}
\item[(i)] For any $\lambda \in  [  -\frac{3}{5},1]$, $\lambda \notin \{  0, \frac{2}{5}\}$, the trivial associative algebra $V_{y^{(1)}_{4A} ( \lambda )}$ is maximal associative.
\item[(ii)] The idempotent $y^{(1)}_{4A} \left( \frac{2}{5} \right)$ is indecomposable with a two-dimensional $0$-eigenspace.
\end{description}
\end{lemma}
\begin{proof}
The equation $h(\lambda)= 1$ has no solutions, while $\overline{h\left( \lambda \right)} = 1$ has the unique solution $\lambda=0$. On the other hand, the equation $h(\lambda)= 0$ has the unique solution $\lambda = \frac{2}{5}$, while $\overline{h\left( \lambda \right)} = 0$ has no solutions. Therefore $d_0(y^{(1)}_{4A} ( \lambda )) = d_1(y^{(1)}_{4A} ( \lambda ))=1$ for any $\lambda \in  [  -\frac{3}{5},1]$, $\lambda \notin \{  0, \frac{2}{5}\}$. Hence, part \textbf{(i)} follows by Corollary \ref{trivial max}. Part \textbf{(ii)} follows by the previous description of the spectrum of $y^{(1)}_{4A} \left( \frac{2}{5} \right)$. 
\end{proof}

Table \ref{Spec4A} gives the spectra of the $\Aut(4A)$-orbits of the non-zero non-identity idempotents of $V_{4A}$, where $y^{(2)}_{4A}$ is defined in \cite[Sec.~3.3]{CR13}.

%SPECTRUM 4A
\begin{table}[!h]
\setlength{\tabcolsep}{10pt}
\renewcommand{\arraystretch}{1.7}
\centering
\begin{tabular}{llll}
\hline
Orbit & Size & Length & Spectrum \\ \hline
$\left[ a_{0}\right] $  & $4$ & $1$ & $\left\{ 0,0,1,\frac{1}{4},\frac{1}{32}\right\}$ \\ 
$[ y^{(1)}_{4A}(\lambda) ]$ & $2$ & $2$ & $\{0, 1, h(\lambda), \overline{h\left( \lambda \right)} \}$ \\
$[ y^{(2)}_{4A} ]$ & $4$ & $\frac{12}{7}$ & $\left\{ 0,1,\frac{1}{14},\frac{5}{14},\frac{6}{7}\right\}$  \\ 
\hline
\end{tabular}
\caption{Spectra of the idempotents of $V_{4A}$.}
\label{Spec4A}
\end{table}

\begin{lemma}\label{sub4A}
The following subalgebras of $V_{4A}$ are maximal associative:
\[ \left\langle \left\langle a_0, \ a_2, \ \id_{4A} - a_0 - a_{2} \right\rangle \right\rangle  \text{ and }  \left\langle \left\langle a_1, \ a_{-1}, \ \id_{4A} - a_1 - a_{-1} \right\rangle \right\rangle.\]
\end{lemma}
\begin{proof}
The sum of the idempotents generating each subalgebra is $\id_{4A}$. Observe that 
\[ y^{(1)}_{4A} \left( \frac{2}{5} \right) = \id_{4A} - a_0 - a_{2} \ \text{ and } \  y^{(1)}_{4A} \left( \frac{2}{5} \right)^{\phi_{4A}} = \id_{4A} - a_1 - a_{-1},   \] 
where $\phi_{4A}$ is the automorphism of $V_{4A}$ that transposes $a_0$ and $a_1$. Therefore, these idempotents are indecomposable by Table \ref{Spec4A} and Theorem \ref{indecomposable}. The result follows by Corollary \ref{cormax}.
\end{proof}

\begin{lemma}\label{Max4A}
The dimension of every associative subalgebra of $V_{4A}$ is at most three.
\end{lemma}
\begin{proof}
Because $\{ y^{(1)}_{4A} ( \lambda ) : \lambda \in [  -\frac{3}{5},1] \}$ is the only infinite family of idempotents in $V_{4A}$ (see \cite[Proposition~3.9]{CR13}), Lemma \ref{infinite} implies that, for every idempotent $x \in V_{4A}$, the algebra $V_0^{(x)}$ has finitely many idempotents. As we see in Table \ref{Spec4A}, $d_0(x) \leq 2$ for every idempotent $x \in V_{4A}$, so the result follows by Lemma \ref{max4}.  
\end{proof}

\begin{lemma}
The Norton-Sakuma algebra of type $4A$ contains infinitely many maximal associative subalgebras. However, it contains only two non-trivial maximal associative subalgebras.
\end{lemma}
\begin{proof}
The first part of this lemma follows by Lemma \ref{infinite}. By Table \ref{Spec4A} and Lemma \ref{semisimple}, only the indecomposable idempotents $[a_0] \cup [ \id_{4A} - a_0 - a_{2}]$ have a two-dimensional $0$-eigenspace, so these are the only idempotents whose trivial associative subalgebra is not maximal. By Lemma \ref{count assoc}, each one of these idempotents is contained in at most one three-dimensional maximal associative subalgebra. Lemma \ref{sub4A} describes such subalgebras. The result follows by Lemma \ref{Max4A}, as there are no associative subalgebras of $V_{4A}$ of dimension greater than $3$. 
\end{proof}

%% CASE 4B %%%%%%%%%%%%%%%%%%%%%%
\subsection{Associative Subalgebras of $V_{4B}$}

Let $\{ a_0, a_1, a_{-1}, a_2, a_{\rho^2}\}$ be the basis of $V_{4B}$ as in \cite[Table~3]{IPSS10}. It is known that $V_{4B}$ contains a Norton-Sakuma subalgebra of type $2A$ with basis $\{a_0, a_{2}, a_{\rho^2} \}$ (see \cite[Lemma 2.20]{IPSS10}). Let $\id_{2A}$ be the identity of this subalgebra. Table \ref{Spec4B} gives the spectra of the $\Aut(4B)$-orbits of half of the non-zero non-identity idempotents of $V_{4B}$, where $y_{4B}$ is defined in \cite[Sec.~3.2]{CR13}.

%SPECTRUM 4B
\begin{table}[!h]
\setlength{\tabcolsep}{10pt}
\renewcommand{\arraystretch}{1.7}
\centering
\begin{tabular}{llll}
\hline
Orbit & Size & Length & Spectrum  \\ \hline
$\left[ a_{0}\right]$ & $4$ & $1$ &  $\left\{ 0,0,1,\frac{1}{4},\frac{1}{32}\right\} $  \\
$\left[ a_{\rho ^{2}}\right]$ & $1$ & $1$ & $\left\{ 0,0,1,\frac{1}{4},\frac{1}{4}\right\}$ \\ 
$\left[ \id_{2A}\right]$ & $2$ & $\frac{12}{5}$ & $\left\{ 0,1,1,1,\frac{1}{4}\right\}$ \\ 
$\left[ \id_{2A}-a_{0}\right]$ & $4$ & $\frac{7}{5}$ & $\left\{ 0,0,1,\frac{3}{4},\frac{7}{32}\right\}$ \\
$\left[ y_{4B}\right]$ & $4$ & $\frac{21}{11}$ & $\left\{ 0,1,\frac{1}{11},\frac{21}{22},\frac{9}{22}\right\}$ \\
\hline
\end{tabular}
\caption{Spectra of the idempotents of $V_{4B}$.}
\label{Spec4B}
\end{table}

\begin{lemma}\label{max4B}
The following subalgebras of $V_{4B}$ are maximal associative:
\begin{align*}
U^{(1)}_{4B} &:= \langle \langle a_0, \ \id_{2A} - a_0, \ \id_{2A}^{\phi_{4B}} - a_{\rho^2} \rangle \rangle, \\
U^{(2)}_{4B} &:= \langle \langle a_{\rho^2}, \ \id_{2A} - a_{\rho^2}, \ \id_{2A}^{\phi_{4B}} - a_{\rho^2} \rangle \rangle,
\end{align*}
where $\phi_{4B} \in \Aut(4B)$ is the automorphism that transposes $a_0$ and $a_1$.
\end{lemma}
\begin{proof}
Because of the relation $\id_{4B} = \id_{2A} + \id_{2A}^{\phi_{4B}} - a_{\rho^2}$, we verify that
\begin{align*}
\id_{4B} &= a_0 + (\id_{2A} - a_0) + (\id_{2A}^{\phi_{4B}} - a_{\rho^2} ), \\
 \id_{4B} & = a_{\rho^2} + (\id_{2A} - a_{\rho^2}) + ( \id_{2A}^{\phi_{4B}} - a_{\rho^2} ).
\end{align*}
By Table \ref{Spec4B} and Theorem \ref{indecomposable}, the idempotents generating $U^{(1)}_{4B}$ and $U^{(2)}_{4B}$ are indecomposable, so the result follows by Corollary \ref{cormax}.
\end{proof}

\begin{lemma}\label{most4B}
The dimension of every associative subalgebra of $V_{4B}$ is at most three.
\end{lemma}
\begin{proof}
Suppose that $U$ is an associative subalgebra of $V_{4B}$ of dimension $k \geq 4$. Without loss of generality, we may assume that $U$ is maximal associative. Let $B:=\{ x_i : 1 \leq i \leq k \}$ be the orthogonal basis of indecomposable idempotents of $U$.  By Theorem \ref{Meyer-Neutsch}, $\sum \limits_{i=1}^{k} x_i = \id_{4B}$. Lemma \ref{linearity} implies that 
\begin{equation}
\sum \limits_{i=1}^{k} l(x_i) = l(\id_{4B}) = \frac{19}{5}, \label{rel4B}
\end{equation}
as may be read from \cite[Table~2]{CR13}. Since there is no pair of orthogonal idempotents of length $1$ in $V_{4B}$, the set $B$ contains at most one idempotent of length $1$. By Table \ref{Spec4B}, $\frac{7}{5}$ is the smallest length different from $1$ of a non-zero idempotent of $V_{4B}$. Therefore,
\[ \sum \limits_{i=1}^{k} l(x_i) \geq 1+3\cdot \frac{7}{5} = \frac{26}{5} > \frac{19}{5}, \]
which contradicts (\ref{rel4B}).
\end{proof}

\begin{lemma}\label{Assoc4B}
The Norton-Sakuma algebra of type $4B$ contains exactly $9$ maximal associative subalgebras; $4$ of these subalgebras are trivial associative while $5$ are three-dimensional. 
\end{lemma}
\begin{proof}
The $4$ trivial maximal associative subalgebras are contained in the orbit $[V_{y_{4B}}]$. With the notation of Lemma \ref{max4B}, the orbits $[U^{(1)}_{4B}]$ and $[U^{(2)}_{4B}]$ contain $4$ and $1$ maximal associative subalgebras, respectively. We shall show that there are no more maximal associative subalgebras. Let $x \in V_{4B}$ be an indecomposable idempotent with $d_0(x) \geq 2$. Let $N_x$ be the number of algebras in $[U^{(1)}_{4B}] \cup [U^{(2)}_{4B}]$ containing $x$. Table \ref{comp} gives the values of $d_0(x)$ and $N_x$ for the orbit representatives of idempotents $x$ with $d_0(x) \geq 2$.

\begin{table}[!h]
\setlength{\tabcolsep}{10pt}
\renewcommand{\arraystretch}{1.7}
\centering
\begin{tabular}{lll|lll}
\hline
Idempotent $x$ & $d_0(x)$ &  $N_x$ & Idempotent $x$ & $d_0(x)$ &  $N_x$\\  \hline
$a_{0}$ & $2$ & $1$ &  $\id_{2A} - a_0$ & $2$ & $1$  \\ 
$ a_{\rho ^{2}}$ & $2$ & $1$ & $\id_{2A} - a_{\rho^2}$ & $3$ & $3$ \\ 
\hline
\end{tabular}
\caption{Values of $d_0(x)$ and $N_x$ for idempotents $x \in V_{4B}$ with $d_0(x) \geq 2$.}
\label{comp}
\end{table}

If $M_x$ is the number of three-dimensional maximal associative subalgebras of $V_{4B}$ containing $x$, Lemma \ref{count assoc} shows that $N_x \leq M_x \leq 2^{d_0(x) -1} - 1$. But Table \ref{comp} shows that $N_x =  2^{d_0(x) - 1} - 1$, for every $x$ with $d_0(x) \geq 2$, so we obtain that $N_x = M_x$. The result follows by Lemma \ref{most4B}.
\end{proof}

%% CASE 5A %%%%%%%%%%%%%%%%%%%%%%
\subsection{Associative Subalgebras of $V_{5A}$}

Let $\{a_{-2}, a_{-1}, a_0, a_1, a_2, w_\rho \}$ be the basis of $V_{5A}$ as in \cite[Table~3]{IPSS10}. Table \ref{Spec5A} gives the spectra of the $\Aut(5A)$-orbits of half of the non-zero non-identity idempotents of $V_{5A}$, where $y^{(i)}_{5A}$, $1 \leq i \leq 3$, are defined in \cite[Sec.~3.4]{CR13}.

%SPECTRUM 5A
\begin{table}[!h]
\setlength{\tabcolsep}{10pt}
\renewcommand{\arraystretch}{1.7}
\centering
\begin{tabular}{llll}
\hline
Orbit & Size & Length & Spectrum \\ \hline
$[ a_{0}] $ & $5$ & $1$ & $\left\{ 0,0,1,\frac{1}{4},\frac{1}{32},\frac{1}{32}\right\}$  \\
$[ y^{(1)}_{5A}]$ & $2$ & $\frac{16}{7}$ & $\left\{ 0,1,\frac{3}{5},\frac{3}{5},\frac{2}{5},\frac{2}{5}\right\}$  \\
$[ y^{(2)}_{5A}]$ & $10$ & $\frac{25}{14}$ & $\left\{ 0,0,1,\frac{5}{64},\frac{57}{64},\frac{3}{8}\right\}$   \\
$[ y^{(3)}_{5A}]$ & $10$ & $\frac{16}{7}$ & $\left\{ 0,1,\frac{7}{8},\frac{3}{5},\frac{2}{5},\frac{1}{8}\right\} $ \\
\hline
\end{tabular}
\caption{Spectra of the idempotents of $V_{5A}$.}
\label{Spec5A}
\end{table}

\begin{lemma}\label{max5A}
The subalgebra $U_{5A} :=\langle \langle a_0, \ y^{(2)}_{5A}, \ (y^{(2)}_{5A})^{\phi_{5A}} \rangle \rangle$ of $V_{5A}$ is maximal associative, where $\phi_{5A}$ is the automorphism of $V_{5A}$ that maps $a_1$ to $a_{2}$ and fixes $a_0$.
\end{lemma}
\begin{proof}
Since
\[ \id_{5A} = a_0 + y^{(2)}_{5A} +  (y^{(2)}_{5A})^{\phi_{5A}}, \]
the result follows by Table \ref{Spec5A}, Theorem \ref{indecomposable} and Corollary \ref{cormax}.
\end{proof}

\begin{lemma}
The Norton-Sakuma algebra of type $5A$ contains exactly $11$ maximal associative subalgebras; $6$ of these algebras are trivial associative while $5$ are three-dimensional.
\end{lemma}
\begin{proof}
By Table \ref{Spec5A} and Corollary \ref{trivial max}, all the trivial maximal associative subalgebras of $V_{5A}$ are contained in $[V_{y_{5A}^{(1)}}] \cup [V_{y_{5A}^{(3)}}]$, so there are $6$ of them. The orbit $[U_{5A}]$ contains $5$ three-dimensional maximal associative subalgebras of $V_{5A}$. Using the same technique as in the proof of Lemma \ref{Assoc4B}, we show that there are no more maximal associative subalgebras of $V_{5A}$.
\end{proof}

%%%%%%%%%%%%% CASE 6A %%%%%%%%%%%%%%%%%%%%
\subsection{Associative Subalgebras of $V_{6A}$}

Let $\{a_{-2}, a_{-1}, a_0, a_1, a_2, a_3, a_{\rho^3}, u_{\rho^2} \}$ be the basis of $V_{6A}$ as in \cite[Table~3]{IPSS10}. Table \ref{Spec6A} gives the spectra of the $\Aut(6A)$-orbits of half of the non-zero non-identity idempotents of $V_{6A}$. In this table, $\id_{2A}$ and $\id_{3A}$ are the identities of the Norton-Sakuma subalgebras of $V_{6A}$ of types $2A$ and $3A$ with bases $\{a_0, a_{3}, a_{\rho^3}\}$ and $\{a_0, a_{2}, a_{-2}, u_{\rho^2}\}$, respectively (see \cite[Lemma 2.20]{IPSS10}). The idempotents $y^{(i)}_{6A}$, $1 \leq i \leq 8$, are defined in \cite[Sec.~3.5]{CR13}. 

%SPECTRUM 6A
\begin{table}[p]
\setlength{\tabcolsep}{10pt}
\renewcommand{\arraystretch}{1.7}
\centering
\begin{tabular}{llll}
\hline
Orbit & Size & Length &  Spectrum \\ \hline
$[ a_{0}]$ & $6$ & $1$ & $\{ 0,0,0,1,\frac{1}{4},\frac{1}{4},\frac{1}{32},\frac{1}{32}\}$   \\ 
$[ a_{\rho ^{3}}]$ & $1$ & $1$ & $\{ 0,0,0,0,1,\frac{1}{4},\frac{1}{4},\frac{1}{4}\}$    \\ 
$[ u_{\rho ^{2}}]$ & $1$ & $\frac{8}{5}$ &  $\{ 0,0,0,1,\frac{1}{3},\frac{1}{3},\frac{1}{3},\frac{1}{3}\}$  \\ 
$[ a_{\rho ^{3}}+u_{\rho ^{2}}]$ & $1$ & $\frac{13}{5}$ &  $\{ 0,1,1,\frac{1}{4},\frac{1}{3},\frac{1}{3},\frac{7}{12},\frac{7}{12}\}$ \\ 
$[ \id_{2A}]$ & $3$ & $\frac{12}{5}$ &  $\{ 0,1,1,1,\frac{1}{4},\frac{3}{10},\frac{3}{10},\frac{1}{20}\}$ \\  
$[ \id_{2A}-a_{0}]$ & $6$ & $\frac{7}{5}$ &  $\{ 0,0,1,\frac{3}{4},\frac{3}{10},\frac{1}{20},\frac{7}{32},\frac{3}{160}\}$   \\ 
$[ \id_{2A}-a_{\rho ^{3}}]$ & $3$ & $\frac{7}{5}$ & $\{ 0,0,0,1,\frac{3}{4},\frac{3}{10},\frac{1}{20},\frac{1}{20}\}$    \\
$[ \id_{3A}]$ & $2$ & $\frac{116}{35}$ &  $\{ 0,1,1,1,1,\frac{5}{14},\frac{5}{14},\frac{5}{14}\}$ \\ 
$[ \id_{3A}-a_{0}]$ & $6$ & $\frac{81}{35}$ & $\{ 0,0,1,\frac{3}{4},\frac{5}{14},\frac{3}{28},\frac{31}{32},\frac{73}{224}\}$  \\ 
$[ \id_{3A}-u_{\rho ^{2}}]$ & $2$ & $\frac{12}{7}$ &  $\{ 0,0,1,\frac{2}{3},\frac{2}{3},\frac{1}{42},\frac{1}{42},\frac{5}{14}\}$ \\ 
$[ y_{3A}]$ & $6$ & $\frac{8}{5}$ & $\{ 0,0,0,1,\frac{1}{3},\frac{1}{3},\frac{1}{16},\frac{13}{16}\}$  \\ 
$[ \id_{3A}-y_{3A}]$ & $6$ & $\frac{12}{7}$ &  $\{ 0,0,1,\frac{2}{3},\frac{3}{16},\frac{5}{14},\frac{1}{42},\frac{33}{112}\}$  \\ 
$[ y^{(1)}_{6A}]$ & $3$ & $\frac{116}{35}$ & $\{ 0,1,1,1,\frac{23}{28},\frac{23}{28},\frac{5}{14},\frac{5}{14}\}$  \\ 
$[ y^{(1)}_{6A}-a_{\rho ^{2}}]$ & $3$ & $\frac{81}{35}$ & $\{ 0,0,1,\frac{4}{7},\frac{5}{14},\frac{3}{4},\frac{3}{28},\frac{23}{28}\}$  \\
$[ y^{(2)}_{6A}]$ & $6$ & $\frac{13}{5}$ &  $\{ 0,1,1,\frac{1}{4},\frac{1}{3},\frac{3}{32},\frac{27}{32},\frac{7}{12}\}$   \\ 
$[ y^{(3)}_{6A}] $ & $6$ & $\frac{12}{7}$ &  $\{ 0,0,1,\frac{2}{3},\frac{5}{14},\frac{1}{42},\frac{1}{112},\frac{85}{112}\}$   \\ 
$[ y^{(4)}_{6A}]$ & $6$ & $\frac{11}{6}$ &  $\{ 0,0,1,\frac{1}{12},\frac{1}{12},\frac{11}{12},\frac{7}{18},\frac{7}{18}\} $   \\ 
$[ y^{(5)}_{6A}]$ & $6$ &  $\frac{97}{30}$ &  $\{ 0,1,1,\frac{5}{6},\frac{31}{45},\frac{2}{15},\frac{29}{30},\frac{7}{18}\}$   \\ 
$[ y^{(6)}_{6A}]$ & $6$ & $\frac{21}{11}$ & $\{ 0,1,\frac{3}{11},\frac{1}{22},\frac{13}{22},\frac{1}{44},\frac{7}{44},\frac{35}{44}\}$   \\ 
$[ y^{(7)}_{6A}]$ & $12$ & $ l( y^{(7)}_{6A})$ &  $d_0(y^{(7)}_{6A}) = d_1(y^{(7)}_{6A}) = 1 $ \\ 
$[ y^{(8)}_{6A}]$ & $12$ & $l( y^{(8)}_{6A})$ & $d_0(y^{(8)}_{6A}) = d_1(y^{(8)}_{6A}) = 1 $ \\ 
\hline
\end{tabular}
\caption{Spectra of the idempotents of $V_{6A}$.}
\label{Spec6A}
\end{table}

\begin{lemma}\label{most6A}
The dimension of every associative subalgebra of $V_{6A}$ is at most three. 
\end{lemma}
\begin{proof}
The result follows by a similar argument to the one used in the proof of Lemma \ref{most4B}, since $\frac{7}{5}$ is the smallest length different from $1$ of a non-zero idempotent of $V_{6A}$ and $l(\id_{6A}) = \frac{51}{10}$. 
\end{proof}

\begin{lemma}
The subalgebras of $V_{6A}$ given in Table \ref{MaxAssoc6A} are maximal associative, where $\phi_{6A} \in \Aut(V_{6A})$ transposes $a_0$ and $a_1$, and $\tau_1$ is the Majorana involution corresponding to $a_1$. 
\end{lemma}
\begin{proof}
This follows by Theorem \ref{indecomposable}, Table \ref{Spec6A} and Corollary \ref{cormax}.
\end{proof}

\begin{table}[h!]
\setlength{\tabcolsep}{10pt}
\renewcommand{\arraystretch}{1.7}
\centering
\begin{tabular}{cc|cc}
\hline
Associative subalgebra $U$ &  $\left\vert [U] \right\vert$ & Associative subalgebra $U$ &  $\left\vert [U] \right\vert$   \\ \hline

$\langle\langle a_{\rho^3}, \ u_{\rho^2}, \ \id_{6A} -a_{\rho^3}-u_{\rho^2} \rangle\rangle$ & $1$ &  $\langle\langle a_{\rho^3}, \ \id_{6A} - \id_{2A}, \ \id_{2A} - a_{\rho^3} \rangle\rangle$ &  $3$   \\

$\langle\langle a_{\rho^3}, \ y^{(1)}_{6A} - a_{\rho^3}, \ \id_{6A} - y^{(1)}_{6A} \rangle\rangle$ &  $3$  & $\langle\langle u_{\rho^2}, \ \id_{3A} - u_{\rho^2}, \  \id_{6A} - \id_{3A} \rangle\rangle$ &  $2$ \\

$\langle\langle a_0, \ \id_{2A} - a_0, \  \id_{6A} - \id_{2A} \rangle\rangle$ &  $6$  &  $\langle\langle a_0, \ \id_{3A} - a_0, \ \id_{6A} - \id_{3A} \rangle\rangle$ &  $6$  \\

$\langle\langle a_0, \ \id_{6A} - y^{(2)}_{6A}, \  (y_{3A})^{\phi_{6A}} \rangle\rangle$ &  $6$ & $\langle\langle y_{3A}, \ \id_{3A} - y_{3A}, \ \id_{6A} - \id_{3A} \rangle\rangle$ &  $6$  \\

$\langle\langle y^{(3)}_{6A}, \ (y_{3A})^{\phi_{6A}}, \ \id_{6A} - y^{(1)}_{6A} \rangle\rangle$ & $6$ & $\langle\langle y^{(4)}_{6A}, \ \id_{6A} - y^{(5)}_{6A}, \ \id_{2A}^{\tau_1} - a_{\rho^3} \rangle\rangle$ & $6$ \\  \hline
\end{tabular}
\caption{Non-trivial maximal associative subalgebras of $V_{6A}$.}
\label{MaxAssoc6A}
\end{table}

\begin{lemma} \label{assoc6A}
The Norton-Sakuma algebra of type $6A$ contains exactly $75$ maximal associative subalgebras; $30$ of these algebras are trivial associative while $45$ are three-dimensional.   
\end{lemma}
\begin{proof}
All the trivial maximal associative subalgebras of $V_{6A}$ are contained in the orbits $[V_{y^{(i)}_{6A}}]$, for $i = 6, 7, 8$, which have sizes $6$, $12$ and $12$, respectively. Using the action of $\Aut(6A)$, Table \ref{MaxAssoc6A} defines $45$ non-trivial maximal associative subalgebras of $V_{6A}$. Using a similar technique as in the proof of Lemma \ref{Assoc4B}, we show that there are no more three-dimensional associative subalgebras in $V_{6A}$. The result follows by Lemma \ref{most6A}. 
\end{proof}

%%%%%%%%%%%%%%%%%%%%%%%%%%%%%%%%%%%%%%%%%%%%%%%%%%%%%%%%

\section{Associative Subalgebras of the Majorana representations of $S_4$ involving $V_{3C}$} \label{S4assoc}

Let $S_4$ be the symmetric group of degree $4$. If $T$ is the set of all involutions in $S_4$, or if $T$ is the set of all transpositions in $S_4$, then $T$ is a generating set of $S_4$ that is closed under conjugation. In both cases, Majorana representations that contain $V_{3C}$ as subalgebra were constructed in \cite[Sec.~4.3--4.4]{IPSS10}. These are called the Majorana representations of $S_4$ of shapes $(2B,3C)$ and $(2A,3C)$, and they are denoted by $V_{(2B,3C)}$ and $V_{(2A,3C)}$, respectively. It was also established in \cite{IPSS10} that $V_{(2B,3C)}$ and $V_{(2A,3C)}$ may be embedded in $V_{\mathbb{M}}$, so they satisfy axiom ${\bf M2}'$. In this last section, we obtain the idempotents and the maximal associative subalgebras of $V_{(2B,3C)}$ and $V_{(2A,3C)}$ using the techniques developed in the previous sections. 

Consider the following Majorana axes corresponding to involutions of $S_4$:
\[ \begin{tabular}{lll}
$a_1 := a_{(1,2)}$, &  $a_2 := a_{(1,3)}$, & $a_3 := a_{(1,4)}$, \\
$a_4 := a_{(2,3)}$, & $a_5 := a_{(2,4)}$, & $a_6 := a_{(3,4)}$, \\
$a_7 := a_{(1,2)(3,4)}$, & $a_8 := a_{(1,3)(2,4)}$, & $a_9 := a_{(1,4)(2,3)}$.
\end{tabular} \]
The algebra $V_{(2B,3C)}$ has dimension $6$ and basis $\{a_i : 1 \leq i \leq 6\}$, while $V_{(2A,3C)}$ has dimension $9$ and basis $\{ a_i : 1 \leq i \leq 9\}$. As $V_{(2B,3C)}$ involves Norton-Sakuma subalgebras of type $2B$, it may not be embedded in $V_{(2A,3C)}$.

%%%%%%%%%%%%%%
\section{Associative Subalgebras of $V_{(2B,3C)}$}

Let $\id_{(2B,3C)}$ be the identity of $V_{(2B,3C)}$. Let $\id_{2B}$ and $\id_{3C}$ be the identities of the Norton-Sakuma subalgebras of $V_{(2B,3C)}$ with bases $\{a_1, a_6 \}$ and $\{a_1, a_2, a_4\}$, respectively. Direct calculations show that the following is an idempotent of $V_{(2B,3C)}$:  
\begin{align*}
y_{(2B,3C)} & := \frac{2}{255} \left[ ( 60 + \sqrt{3570}) a_1  +  ( 60 - \sqrt{3570}) a_6 \right] \\
& \ \ \ + \frac{8}{255} \left[ ( 15 +  \sqrt{255}) ( a_2 +a_5 ) + ( 15 -  \sqrt{255}) ( a_3 +a_4 ) \right].
\end{align*}

By solving the system of quadratic equations determined by the relation $v\cdot v - v =0$, we verify in \cite{MAP} that the algebra $V_{(2B,3C)}$ has exactly $64$ idempotents. Table \ref{Spectrum(2B,3C)} contains the spectra of the $\Aut(V_{(2B,3C)})$-orbits of half of the non-zero non-identity idempotents of $V_{(2B,3C)}$.

\begin{table}[!h]
\setlength{\tabcolsep}{10pt}
\renewcommand{\arraystretch}{1.7}
\centering
\begin{tabular}{llll}
\hline
Orbit & Size & Length & Spectrum \\ 
\noalign{\smallskip}\hline\noalign{\smallskip}
$\left[ a_1 \right]$ & $6$ & $1$ & $\{0,0,0,1,\frac{1}{32},\frac{1}{32}\}$ \\
$\left[a_1 +  a_6\right] $ & $3$ & $2$ & $\{0,1,1,\frac{1}{16}, \frac{1}{32}, \frac{1}{32}\}$  \\
$\left[ \id_{3C}\right] $ & $4$ & $\frac{32}{11}$ & $\{0,1,1,1,\frac{1}{22}, \frac{1}{22}\}$ \\
$\left[ \id_{3C} - a_1 \right] $ & $12$ & $\frac{21}{11}$ & $\{0,0,1,\frac{1}{22},\frac{31}{32},\frac{5}{352}\}$  \\
$\left[ y_{(2B,3C)}\right] $ & $12$ & $\frac{48}{17}$ & $d_0(y_{(2B,3C)}) =  d_1(y_{(2B,3C)}) = 1$ \\ 
\hline
\end{tabular}
\caption{Spectra of idempotents of $V_{(2B,3C)}$.}
\label{Spectrum(2B,3C)}
\end{table}

\begin{lemma}
The following subalgebras of $V_{(2B,3C)}$ are maximal associative:
\begin{align*}
U_{(2B,3C)}^{(1)} & := \langle\langle a_1, \ \id_{3C} - a_1, \ \id_{(2B,3C)} - \id_{3C} \rangle\rangle, \\
U_{(2B,3C)}^{(2)} &:= \langle\langle a_1, \ a_6, \ \id_{(2B,3C)} - a_1 - a_6 \rangle\rangle.
\end{align*}
\end{lemma}
\begin{proof}
This follows by Theorem \ref{indecomposable}, Table \ref{Spectrum(2B,3C)} and Corollary \ref{cormax}.
\end{proof}

\begin{lemma}
The dimension of every associative subalgebra of $V_{(2B,3C)}$ is at most three.
\end{lemma}
\begin{proof}
The result follows by a similar argument to the one used in the proof of Lemma \ref{most4B}, since $\frac{21}{11}$ is the smallest length different from $1$ of a non-zero idempotent of $V_{(2B,3C)}$ and  $l(\id_{(2B,3C)}) = \frac{96}{17}$. 
\end{proof}

\begin{lemma}
The algebra $V_{(2B,3C)}$ contains exactly $21$ maximal associative subalgebras; $6$ of these algebras are trivial associative while $15$ are three-dimensional.  
\end{lemma}
\begin{proof}
The orbit $[ V_{y_{(2B,3C)}}]$ contains all the trivial maximal associative subalgebras of $V_{(2B,3C)}$. The orbits $[U_{(2B,3C)}^{(1)}]$ and $[U_{(2B,3C)}^{(2)}]$ contain $12$ and $3$ maximal associative subalgebras, respectively. The result follows by a similar argument to the one used in the proof of Lemma \ref{assoc6A}.    
\end{proof}

%%%%%%%%%%%%%%
\section{Associative Subalgebras of $V_{(2A,3C)}$}

Let $\id_{(2A,3C)}$ be the identity of $V_{(2A,3C)}$. Let $\id_{2A}$, $\id_{2A}^\prime$, $\id_{3C}$ and $\id_{4B}$ be the identities of the Norton-Sakuma subalgebras of $V_{(2A,3C)}$ with bases $\{a_1, a_6, a_7 \}$, $\{a_7, a_8, a_9 \}$, $\{a_1, a_2, a_4\}$ and $\{a_1, a_6, a_7, a_8, a_9\}$, respectively.

By solving the system of quadratic equations determined by the relation $v \cdot v - v = 0$ in \cite{MAP}, we found that $V_{(2A,3C)}$ has precisely $512$ idempotents. We obtained radical expressions for $464$ of these idempotents, and numerical approximations for the remaining ones. The coefficients of these idempotents are too lengthy to be displayed here.

\begin{table}[p]
\centering
\begin{tabular}{lllll}
\hline\noalign{\smallskip}
Orbit & Size & Length & $d_0(x)$ & $d_1(x)$  \\ 
\noalign{\smallskip}\hline\noalign{\smallskip}
$[a_1]$ & $6$ & $1$ & $4$ & $1$  \smallskip \\

$[a_7]$ & $3$ & $1$ & $4$ & $1$ \smallskip \\

$[\id_{2A}]$ & $3$ & $\frac{12}{5}$ & $2$ & $2$ \smallskip \\

$[\id_{2A} - a_1]$ & $6$ & $\frac{7}{5}$ & $3$ & $1$ \smallskip \\

$[\id_{2A} - a_7]$ & $3$ & $\frac{7}{5}$ & $4$ & $1$ \smallskip \\

$[\id_{2A}^\prime]$ & $1$ & $\frac{12}{5}$  & $3$ & $3$ \smallskip \\

$[\id_{2A}^\prime - a_7]$ & $3$ & $\frac{7}{5}$ & $5$ & $1$  \smallskip \\

$[\id_{3C}]$ & $4$ & $\frac{32}{11}$ & $1$ & $3$  \smallskip \\

$[\id_{3C} - a_1]$ & $12$ & $\frac{21}{11}$ & $2$ & $1$  \smallskip \\

$[\id_{4B}]$ & $3$ & $\frac{19}{5}$ & $1$ & $5$  \smallskip \\

$[\id_{4B} - a_1]$ & $6$ & $\frac{14}{5}$ & $2$ & $2$ \smallskip \\

$[\id_{4B} - a_8]$ & $6$ & $\frac{14}{5}$ & $2$ & $2$ \smallskip \\

$[\id_{4B} - a_7]$ & $3$ & $\frac{14}{5}$ & $2$ & $2$  \smallskip \\

$[\id_{4B} - \id_{2A} + a_1]$ & $6$ & $\frac{12}{5}$ & $2$ & $2$ \smallskip \\

$[\id_{4B} - \id_{2A}^\prime + a_8]$ & $6$ & $\frac{12}{5}$ & $2$ & $2$   \smallskip \\

$[y_{4B}]$ & $12$ & $\frac{21}{11}$ & $2$ & $1$ \smallskip \\

$[\id_{4B} - y_{4B}]$ & $12$ & $\frac{104}{55}$ & $2$ & $1$    \smallskip \\

$[x_1]$ & $12$ & $\frac{4}{5}(4-3\delta_1)$ & $1$ & $1$   \smallskip \\

$[x_2]$ & $4$ & $\frac{4}{5}(4-3\delta_1)$ & $1$ & $1$   \smallskip \\

$[x_3]$ & $12$ & $\frac{1}{10}(27+499\delta_2)$ & $2$ & $1$  \smallskip \\

$ [ x_3^{c_{\delta_2}} ]$ & $12$ & $\frac{1}{10}(27-499\delta_2))$ & $2$ & $1$  \smallskip \\

$[x_4]$ & $12$ & $\frac{4}{5}(4-\delta_3)$ & $1$ & $1$  \smallskip \\

$[ x_4 ^{c_{\sqrt{3}}}]$ & $12$ & $\frac{4}{5}(4-\delta_3)$ & $1$ & $1$ \smallskip \\

$[x_5]$ & $12$ & $\frac{4}{5}(4+\delta_3)$ & $1$ & $1$  \smallskip \\

$[x_5^{c_{\sqrt{11}} }]$ & $12$ & $\frac{4}{5}(4+\delta_3)$ & $1$ & $1$ \smallskip \\

$[x_6]$ & $12$ & $\frac{1}{2}(5+33\delta_4)$ & $2$ & $1$ \smallskip \\

$[ x_6^{c_{\delta_4}} ]$ & $12$ & $\frac{1}{2}(5-33\delta_4)$ & $2$ & $1$  \smallskip \\
 
$[x_{8}]$ & $24$ & $l(x_8)$ & $1$ & $1$  \smallskip \\

$[x_{9}]$ & $24$ & $l(x_9)$ & $1$ & $1$  \\ 
\noalign{\smallskip}\hline
\end{tabular}
\caption{Idempotents of $V_{(2A,3C)}$.}
\label{Idem(2A,3C)}
\end{table}

Table \ref{Idem(2A,3C)} contains the values of $d_0(x)$ and $d_1(x)$ of the $\Aut(V_{(2A,3C)})$-orbits of half the non-zero non-identity idempotents of $V_{(2A,3C)}$, where $c_{\sqrt{d}}$ denotes the conjugation automorphism of $\mathbb{Q}(\sqrt{d})$, and 
\[ \delta_1 := \frac{\sqrt{19009}}{19009}, \ \delta_2 := \frac{\sqrt{1621}}{1621} \text{ and } \delta_3 := \frac{\sqrt{5321}}{5321}. \]

\begin{lemma}
The subalgebras of $V_{(2A,3C)}$ given by Table \ref{MaxAssoc2A3C} are maximal associative.
\end{lemma}
\begin{proof}
This may be verified directly using Theorem \ref{indecomposable}, Table \ref{Idem(2A,3C)} and Corollary \ref{cormax}.
\end{proof}

\begin{lemma}
The dimension of every associative subalgebra of $V_{(2A,3C)}$ is at most four.
\end{lemma}
\begin{proof}
The result follows by a similar argument to the one used in the proof of Lemma \ref{most4B}, since $\frac{7}{5}$ is the smallest length different from $1$ of a non-zero idempotent of $V_{(2A,3C)}$ and  $l(\id_{(2A,3C)}) = \frac{32}{5}$. 
\end{proof}

\begin{table}[!h]
\setlength{\tabcolsep}{10pt}
\renewcommand{\arraystretch}{1.7}
\centering
\begin{tabular}{lcc}
\hline
Associative subalgebra $U$ &  $\left\vert [ U ] \right\vert$ & $\dim(U)$  \\ \hline

$\langle\langle a_1, \ \id_{2A}^\prime - a_7, \ \id_{2A} - a_1, \ \id_{(2A,3C)} - \id_{4B} \rangle\rangle$ & $6$ & $4$ \\

$\langle\langle a_7, \ \id_{2A}^\prime - a_7, \ \id_{2A} -  a_7, \ \id_{(2A,3C)} - \id_{4B}  \rangle\rangle$ &  $3$ & $4$   \\

$\langle\langle a_8, \ \id_{2A}^\prime - a_8, \  \id_{2A} - a_7, \ \id_{(2A,3C)} - \id_{4B} \rangle\rangle$ &  $3$ & $4$   \\

$\langle\langle a_9, \ \id_{2A}^\prime - a_9, \  \id_{2A} - a_7, \ \id_{(2A,3C)} - \id_{4B} \rangle\rangle$ &  $3$ & $4$   \\

$\langle\langle a_1, \ \id_{3C} - a_1, \ \id_{(2A,3C)} - \id_{3C} \rangle\rangle$ &  $12$ & $3$   \\

$\langle\langle y_{4B}, \ \id_{(2A,3C)} - \id_{4B}, \ \id_{4B} - y_{4B} \rangle\rangle$ &  $3$ & $3$   \\

$\langle\langle a_1, \ x_3,\ x_3^{c_{\delta_2}}  \rangle\rangle$ &  $12$ & $3$   \\

$\langle\langle \id_{2A}^\prime - a_7, \ x_6, \ x_6^{c_{\delta_4}}   \rangle\rangle$ &  $12$ & $3$   \\ \hline
\end{tabular}
\caption{Non-trivial maximal associative subalgebras of $V_{(2A,3C)}$.}
\label{MaxAssoc2A3C}
\end{table}

\begin{lemma}
The algebra $V_{(2A,3C)}$ contains exactly $166$ maximal associative subalgebras; in particular, exactly $54$ of these subalgebras are non-trivial, as given by Table \ref{MaxAssoc2A3C}. 
\end{lemma}
\begin{proof}
In order to find all the maximal associative subalgebras of $V_{(2A,3C)}$, we consider the indecomposable idempotents of $V_{(2A,3C)}$. 

As may be seen from Table \ref{Idem(2A,3C)}, there are $112$ indecomposable idempotents $x \in V_{(2A,3C)}$ with $d_0(x)=1$, so there are $112$ trivial maximal associative subalgebras. 

If $x \in V_{(2A,3C)}$ is an indecomposable idempotent with $d_0(x) =2$, then $x$ is in one of the orbits 
\[ [\id_{3C} - a_1], \ [y_{4B}], \ [\id_{4B} - y_{4B}], \ [x_3], \  [ x_3^{c_{\delta_2}} ], \ [x_6] \ \text{ or } \ [ x_6^{c_{\delta_4}} ].\]
Table \ref{MaxAssoc2A3C} shows that each one of these idempotents is contained in a three-dimensional maximal associative subalgebra, and the proof of Lemma \ref{max4} shows that such idempotents cannot be contained in any other maximal associative subalgebra. 

If $x \in V_{(2A,3C)}$ is an indecomposable idempotent with $d_0(x) >2$, then $x$ is contained in one of the orbits 
\[ [a_1], \ [a_7], \ [\id_{2A} - a_1], \ [\id_{2A} - a_7], \ [\id_{2A}^\prime - a_7], \ [\id_{(2A,3C)} -  \id_{3C}] \ \text{ or } \ [\id_{(2A,3C)} - \id_{4B}]. \]
Direct calculations allow us to conclude that all the maximal associative subalgebras containing these idempotents are given by Table \ref{MaxAssoc2A3C}. 
\end{proof}

\noindent \textbf{Acknowledgments} \newline
I gratefully thank my PhD supervisor Alexander A. Ivanov for all his helpful advice and comments. My thanks also to Felix Rehren for sharing with me his interesting thoughts about associative subalgebras and idempotents, and to the anonymous referee of this paper for all the insightful comments. 

%%%%%%%%BIBLIOGRAPHY%%%%%%%%%%%%%%%%%%%%%%%%%%%%%%%%
\small
\bibliographystyle{elsarticle-num}

\begin{thebibliography}{99}

\bibitem[CR13]{CR13}
A. Castillo-Ramirez, `Idempotents of the Norton-Sakuma algebras', \emph{J. Group Theory} 16.3 (2013) 419--444.

\bibitem[C84]{C84}
J.H. Conway, `A simple construction for the Fischer-Griess Monster group', \emph{Invent. Math.} 79, (1984) 513--540.

\bibitem[DGH98]{DGH98}
C. Dong, R.L. Griess and G. Hohn, `Framed Vertex Operator Algebras, Codes and the Moonshine Module', \emph{Commun. Math. Phys.} 193 (1998) 407--448.

\bibitem[DLMN98]{DLMN98}
C. Dong, H. Li, G. Mason and S. P. Norton, `Associative subalgebras of the Griess algebra and related topics', in \emph{The Monster and Lie Algebras}, Ohio State Univ. Math. Res. Inst. Publ. 7 (1998) 27--42. 

\bibitem[G82]{G82} 
R. L. Griess, `The Friendly Giant', \emph{Invent. Math.}, 69 (1982) 1--102.

\bibitem[GMS89]{GMS89}
R. L. Griess, U. Meierfrankenfeld and Y. Segev, `A uniqueness proof for the Monster', \emph{Ann. of Math. (2)} 130 (1989), no.3, 567--602.

\bibitem[Iv09]{I09}
A. A. Ivanov, \emph{The Monster Group and Majorana Involutions}, Cambridge Univ. Press, Cambridge, Cambridge Tracts in Mathematics 176 (2009).

\bibitem[IPSS10]{IPSS10}
A. A. Ivanov, D. V. Pasechnik, \'{A}. Seress and S. Shpectorov, `Majorana representations of the symmetric group of degree $4$', \emph{J. Algebra} 324 (2010) 2432--2463.

\bibitem[Iv11a]{I11a}
A. A. Ivanov, `On Majorana Representations of $A_{6}$ and $A_{7}$', \emph{Comm. Math. Phys.} 307 (2011) 1--16.

\bibitem[Iv11b]{I11b}
A. A. Ivanov, `Majorana Representations of $A_{6}$ involving $3C$-algebras', \emph{Bull. Math. Sci.} 1 (2011) 356--378.

\bibitem[IS12]{IS12}
A. A. Ivanov and \'{A}. Seress, `Majorana Representations of $A_{5}$', \emph{Math. Z.} 272 (2012) 269--295.

\bibitem[IS]{IS}
A. A. Ivanov and S. Shpectorov, `Majorana Representations of $L_{3}( 2) $', \emph{Adv. Geom.}, 14 (2012) 717--738.

\bibitem[LYY05]{LYY05}
C.H. Lam, H. Yamada and H. Yamauchi, `McKay's observation and vertex operator algebras generated by two conformal vectors of central charge 1/2', \emph{Int.  Math. Res. Papers}, 3 (2005) 117--181. 

\bibitem[MAP]{MAP}
Maplesoft, \emph{Maple 16.00 - The Essential Tool for Mathematics and Modeling}, Licensed to: Imperial College Centre for Computing Services (2012).

\bibitem[MN93]{MN93}
W. Meyer and W. Neutsch, `Associative subalgebras of the Griess algebra', \emph{J. Algebra}, 158 (1993) 1--17.

\bibitem[Mi96]{Mi96}
M. Miyamoto, `Griess algebras and conformal vectors in vertex operator algebras', \emph{J. Algebra}, 179 (1996) 523--548.

\bibitem[N96]{N96}
S. P. Norton, `The Monster algebra: some new formulae', in \emph{Moonshine, the Monster and Related Topics}, \emph{Contemp. Math.}, 193, AMS, Providence, RI, (1996) 297--306.

\bibitem[R08]{R08}
S. Roman, `Advanced Linear Algebra', Springer 135, \emph{Graduate Texts in Mathematics} (2008).

\bibitem[Sa07]{S07}
S. Sakuma, `$6$-Transposition property of $\tau $-involutions of vertex operator algebras', \emph{Int. Math. Res. Not. IMRN}, 9 (2007) 19.

\bibitem[Se12]{Se12}
\'A. Seress, `Construction of $2$--closed ${M}$--representations', \emph{Proc. Int. Sympos. Symbolic Algebraic Comput.}, Amer. Math. Soc., New York, (2012) 311--318.
\end{thebibliography}

\end{document}